\documentclass[10pt]{amsart}
\usepackage{amsfonts,amsmath,amsthm,amssymb,enumerate,color,mathrsfs}
\usepackage[makeroom]{cancel}
\usepackage[color=blue!70!cyan!10!white]{todonotes}

\numberwithin{equation}{section}
\newcommand{\beq}{\begin{equation}}
\newcommand{\eeq}{\end{equation}}
\newcommand{\bea}{\begin{eqnarray}}
\newcommand{\eea}{\end{eqnarray}}
\newcommand{\beas}{\begin{eqnarray*}}
\newcommand{\eeas}{\end{eqnarray*}}

\newtheorem{theorem}{Theorem}[section]

\newtheorem{proposition}[theorem]{Proposition}

\newtheorem{corollary}[theorem]{Corollary}
\newtheorem{lemma}[theorem]{Lemma}
\newtheorem{remark}[theorem]{Remark}
\newtheorem{example}[theorem]{Example}
\newtheorem{examples}[theorem]{Examples}
\newtheorem{foo}[theorem]{Remarks}

\newcommand{\calE}{\mathcal{E}}

\newcommand{\Ho}{\mathcal H}
\newcommand{\V}{\mathcal V}
\newcommand{\ch}{\mathcal H}
\newcommand{\calH}{\mathcal H}
\newcommand{\calV}{\mathcal V}

\DeclareMathOperator{\spn}{span}

\newcommand{\ve}{\varepsilon}

\DeclareMathOperator{\tr}{tr}

\newcommand{\M}{\mathbb M}

\DeclareMathOperator{\Sec}{Sec}

\newcommand{\hptr}{/ \! \hat{/}}

\newcommand{\Sas}{\mathrm{Sas}}
\newcommand{\Rie}{\mathrm{Rie}}

\title[Variations of the sub-Riemannian distance on Sasakian manifolds]{Variations of the sub-Riemannian distance on Sasakian manifolds with applications to coupling}

\author[Baudoin, Grong, Neel, Thalmaier]{Fabrice Baudoin, Erlend Grong, Robert Neel, and Anton Thalmaier}

\address{
Fabrice Baudoin, Department of Mathematics, University of Connecticut,
   341 Mansfield Road, Storrs, CT  06269-1009, USA}
\email{fabrice.baudoin@uconn.edu}

\address{
Erlend Grong, University of Bergen, Department of Mathematics, P.O. Box 7803, 5020 Bergen, Norway}
\email{erlend.grong@uib.no}

\address{
Robert Neel, Department of Mathematics, Lehigh University, Bethlehem, PA, USA}
\email{robert.neel@lehigh.edu}

\address{
Anton Thalmaier, Mathematics Research Unit, University of Luxembourg, L-4364 Esch-sur-Alzette, Luxembourg}
\email{anton.thalmaier@uni.lu}

\thanks{The first author supported in part by NSF Grant DMS 1901315.}
\thanks{The second author supported by project GeoProCo from the Trond Mohn Foundation - Grant TMS2021STG02.}
\thanks{This third author was partially supported by a grant from the Simons Foundation (\#524713 to Robert Neel)}
\thanks{The fourth author supported by FNR Luxembourg: OPEN project GEOMREV O14/7628746.}


\begin{document}

\maketitle

\begin{abstract}
On Sasakian manifolds with their naturally occurring sub-Riemannian structure, we consider parallel and mirror maps along geodesics of a taming Riemannian metric. We show that these transport maps have well-defined limits outside the sub-Riemannian cut-locus. Such maps are not related to parallel transport with respect to any connection. We use this map to obtain bounds on the second derivative of the sub-Riemannian distance. As an application, we get some preliminary result on couplings of sub-Riemannian Brownian motions.
\end{abstract}

\section{Introduction}
Studying couplings on Riemannian manifolds has been a successful method of obtaining functional inequalities related to local geometric invariants, see, for example, \cite{WangBook} and the references therein. A key preliminary step is to first establish comparison results, in particular related to derivatives of the Riemannian distance. There has been much interest in finding functional inequalities involving the sub-Laplacian which depends on sub-Riemannian geometric identities; see \cite{ABR2,BR2,Bau,BB,BBG,BG,BK,BKW,GT1,GT2,GT3} for examples of progress in this direction. Recall that a sub-Riemannian manifold is a triple $(\M, \calH, g_{\calH})$, where $\calH$ is a bracket-generating subbundle of the tangent bundle $T\M$ of the manifold $\M$ and $g_{\calH}$ is a smoothly varying inner product on $\calH$. Such manifolds are related to second-order operators with a positive semi-definite symbol $\Delta_{\calH}$ called a sub-Laplacian. These operators are not elliptic, but both the sub-Laplacian and its heat operator will be hypoelliptic by H\"ormander's classical result in \cite{Hormander}. The interest in such results comes not only from the applications to hypoelliptic PDEs but also in deepening the understanding of sub-Riemannian geometry itself. Recently, there have been several comparison results using curvature bounds to control the Hessian and sub-Laplacian of the sub-Riemannian distance function; see \cite{AL1,AL2,Lee,LL,BGKT,BGMR19} for examples, with an application found in \cite{BGKNT}. We want to emphasize the previous results found in \cite{BGKT,BGMR19} obtain Hessian comparison theorems by considering the index forms and Jacobi fields of a model taming Riemannian metric $g_\ve$ that approaches a sub-Riemannian metric as $\ve \to 0$. We will expand the application of this method in the current paper.

The aim of this paper is to pave the way for constructing couplings of processes whose infinitesimal generator is a sub-Laplacian. Such couplings have already been studied in \cite{BonnJuil} for the case of the Heisenberg group (see also \cite{BGM} for non-Markovian couplings on the Heisenberg group). We will begin with the study of sub-Riemannian manifolds that can be obtained from Sasakian contact manifolds, where the horizontal bundle $\calH$ is the kernel of the contact form. We consider a taming Riemannian metric $g_\ve$ that converge to the sub-Riemannian metric $(\calH, g_{\calH})$ defined such that the Reeb vector field remains orthogonal to $\calH$. We then study derivatives of the Riemannian distance function $d_{\ve}(\gamma_{1,\ve}(t), \gamma_{2,\ve}(t))$, where $\gamma_{1,\ve}(t)$ and $\gamma_{2,\ve}(t)$ are two different $g_\ve$-geodesics. We construct a parallel and mirror coupling between such choices of geodesics, and show that these coupling maps have a well defined limit as we let the Riemannian geodesics approach their sub-Riemannian counterpart. We note however that the limiting parallel map is not the parallel transport of any affine connection. Using curvature bounds, we are able to control the first and second derivative $d_{\ve}(\gamma_{1,\ve}(t), \gamma_{2,\ve}(t))$ such that we have a well defined limits as $\ve\to 0$, giving a sub-Riemannian result. It is a basic feature of the subject that such derivatives of $d_{\ve}(\gamma_{1,\ve}(t), \gamma_{2,\ve}(t))$ are only well defined at points that are not in the cut locus. We restrict our attention to this situation in the present paper, since it forms the geometric foundation to analyze coupled Brownian motions and is also of independent interest. The extension to globally defined couplings and their applications will be the subject of future research.

The structure of the paper is as follows. In Section~\ref{sec:Preliminaries} we consider the necessary preliminaries of the paper, such as index forms on Riemannian manifolds, sub-Riemannian geometry, and cut loci of both sub-Riemannian and Riemannian manifolds. We will also describe how sub-Riemannian manifolds can be considered as a limit of Riemannian manifolds. In Section~\ref{sec:MainSection}, we consider our main results on sub-Riemannian Sasakian manifolds. We define parallel and mirror maps in Section~\ref{Sect:ParallelAndMirror} that preserves both the Riemannian structure $g_\ve$ as well as the horizontal and vertical bundle, and we show that these have a well-defined limits. We will use these maps to get a short-time expansion of the sub-Riemannian distance along coupled geodesics in Theorem~\ref{theorem:Expansion}. We finally use these results in Section~\ref{sec:Coupling} to get some bounds for the distances between coupled sub-Riemannian Brownian motions that hold up to the cut locus.


\section{Preliminaries} \label{sec:Preliminaries}
\subsection{Variation of Riemannian distance and the index form}
Let $(\M,g)$ be a Riemannian manifold with corresponding Riemannian distance $d_{g}$. For any $x \in \M$, define $r_{g,x}(y) = d_g(x,y)$. The cut locus $\mathbf{Cut}_g (x)$  is defined such that $ y \in \M \setminus \mathbf{Cut}_g(x)$ if there exists a unique, non-conjugate, length-minimizing geodesic from~$x$ to~$y$ relative to~$g$. The global cut-locus of~$\M$ is defined by
\[
\mathbf{Cut}_g (\M)=\left\{ (x,y) \in \M \times \M,\  y \in \mathbf{Cut}_g (x) \right\},
\]
and note that it is symmetric (that is, $(x,y)\in\mathbf{Cut}_g (\M)$ if and only if $(y,x) \in\mathbf{Cut}_g (\M)$).
\begin{lemma}[\cite{A}, \cite{RT}]\label{cutlocus}
The following statements hold:
 \begin{enumerate}[\rm (a)]
\item The set $\M \setminus \mathbf{Cut}_g (x)$ is open and dense in $\M$ for any $x \in M$.
\item The function $(x,y) \to d_g (x,y)^2$ is smooth on $\M \times \M \setminus \mathbf{Cut}_g (\M)$.
\end{enumerate}
\end{lemma}

Let $\gamma:[0,r] \to \M$ be a geodesic of $g$, parametrized by arc length. For vector fields along~$\gamma$, we define a symmetric map $I_{g} = I_{g,\gamma}: (X,Y) \mapsto I_g(X,Y)$ by
$$I_g(Y,Y) = \int_0^r \left( \| \nabla^g_{\dot \gamma} Y(t) \|^2_{g} + \langle R^g(\dot \gamma(t), Y(t)) \dot \gamma(t) , Y(t) \rangle_g \right) dt.$$
where $\nabla^g$ is the Levi-Civita connection of $g$ and $R^g$ is its curvature. The following result can be found in e.g. \cite{Wang}.
\begin{lemma} \label{lemma:DistanceVariation}
Let $\gamma_1$ and $\gamma_2$ be two geodesics with initial velocity
$$\gamma_1'(0) = v \in T_xM, \qquad \gamma_2'(0)= w \in T_y M.$$
Assume that $y$ is not in the cut locus of $x$ and let $\gamma_{x,y}:[0,r] \to M$ be the unique unit speed geodesic from $x$ to $y$ with $r = d_g(x,y)$. Let $Y$ be the Jacobi field along $\gamma_{x,y}$ satisfying $Y(0) = v$ and $Y(r) = w$ and $Y^\perp(t)$ is the projection to the orthogonal complement of $\dot \gamma_{x,y}$. Then
\begin{align*}  \tfrac{d}{dt} d_g( \gamma_1(t), \gamma_2(t) ) |_{t=0} & = \langle w, \dot \gamma_{x,y}(r) \rangle - \langle v, \dot \gamma_{x,y}(0) \rangle; \\
\tfrac{d^2}{dt^2} d_g( \gamma_1(t), \gamma_2(t) ) |_{t=0} & = I_g(Y^\perp,Y^\perp) = I_g(Y,Y) - \frac{1}{r} \left(\langle w, \dot \gamma_{x,y}(r) \rangle_g - \langle v, \dot \gamma_{x,y}(0) \rangle_g \right)^2 .
\end{align*}
\end{lemma}
This result gives us a way of computing derivatives of the distance if we know the Jacobi fields along the curve. We can also estimate the derivate using the following Lemma, see e.g.~\cite[Corollary~6.2]{Lee97} and \cite[Theorem 1.1.11]{Wang}.
\begin{lemma}[Index Lemma] \label{lemma:Index}
Let $\gamma: [0,r] \to \M$ be a minimizing geodesic and let $I_g = I_{g,\gamma}$ be its index form. Let $Y$ be a Jacobi field and let $X$ be any vector field along $\gamma$ with $X(0) = Y(0)$ and $Y(r) = X(r)$. Then
$$I_g(Y,Y) \leq I_g(X,X).$$
\end{lemma}
We note that $I_g$ can be rewritten using other connections rather than the Levi-Civita connection. Let $\nabla'$ be an arbitrary connection. We say that $\nabla'$ is \emph{compatible with} $g$ if $\nabla' g = 0$. A compatible connection has the same geodesics as the Levi-Civita connection if and only if its torsion $T^\prime$ is \emph{skew-symmetric}, meaning that $\langle T^\prime(u,v), w \rangle_g = -\langle v, T^\prime(u,w) \rangle_g$, $u,v,w \in T\M$. Alternatively, a compatible connection $\nabla^\prime$ is skew-symmetric if and only if \emph{its adjoint connection} $\hat \nabla^\prime$ defined by $\hat \nabla^\prime_X Y = \nabla_X^\prime Y - T^\prime(X,Y)$ is compatible with $g$ as well.

Let $\nabla'$ be a compatible connection with skew symmetric torsion and with curvature $R'$. In this case, we can write
$$I_g(Y,Y) = \int_0^r \left( \langle \nabla'_{\dot \gamma} Y, \hat \nabla'_{\dot \gamma} Y \rangle_{g} + \langle R'(\dot \gamma, Y) \dot \gamma , Y \rangle_g \right) dt.$$
For details, see \cite{BGKT,BGMR19}. In what follows, we will focus on the choice of connection that will preserve a given decomposition of the tangent bundle $T\M = \Ho \oplus \V$.

\subsection{Sub-Riemannian manifolds}\label{Sect:sRManifolds}
A sub-Riemannian manifold is a triple $(\M , \Ho, g_\Ho)$, where $\M$ is a connected manifold, $\Ho$ is a subbundle of the tangent bundle $T\M$ and $g_\Ho = \langle \cdot , \cdot \rangle_{\calH}$ is a metric tensor on $\Ho$. The subbundle $\Ho$ is assumed to be \emph{bracket-generating}, meaning that $T\M$ is spanned by sections of $\Ho$ and their iterated brackets. This assumption is a sufficient condition for connectivity of any pair of points by a \emph{horizontal curve}, that is, an absolutely continuous curve $\gamma$ which is tangent to $\Ho$ almost everywhere. For such a curve $\gamma:[0,t_1] \to \M$, we can define its length as $L(\gamma) = \int_0^{t_1} \langle \dot \gamma(t), \dot \gamma(t) \rangle^{1/2}_{\calH} \, dt$. Subsequently, we can define a distance on $\M$,
$$d_0(x, y) = \inf_{\gamma} \Big\{ L(\gamma) \, : \gamma(0) = x, \gamma(t_1) = y, \text{$\gamma$ a horizontal curve} \Big\},$$
which induces the same topology as the manifold topology.

The exponential map on a sub-Riemannian manifold is defined as follows.  Let $\pi: T^*\M \to \M$ denote the canonical projection of the cotangent bundle. From the sub-Riemannian structure $(\calH, g)$, we have a corresponding vector bundle morphism $\sharp_{0}: T^* \M \to \Ho$ uniquely defined by
$$p(v) = \langle \sharp_0 p, v \rangle_\Ho, \qquad \text{for any $p \in T^*_x \M$, $v \in \Ho$, $x \in \M$.}$$
Define a Hamiltonian function $H: T^* \M \to \mathbb{R}$ by $H(p) = \frac{1}{2} \langle \sharp_0 p, \sharp_0 p \rangle_{\Ho}$. Let $\vec{H}$ denote the corresponding Hamiltonian vector field with local flow $t \mapsto e^{t\vec{H}}$. For any $p \in T^*M$, we write
$$\exp_0(tp) = \pi(e^{t\vec{H}}(p)).$$
for any sufficiently small $t$ such that the above expression is well defined. We remark that if $\lambda(t) = e^{t\vec{H}}
(p)$ and $\gamma(t) = \exp_0(tp) = \pi(\lambda(t))$, then $\dot \gamma(t) = \sharp_0 \lambda(t)$ and we have that the speed is constant and equal to $\langle \sharp_0 p, \sharp_0 p\rangle^{1/2}$. We say that $\gamma(t) = \exp_0(tp)$ is \emph{the normal geodesic} with initial covector $p \in T^*M$. Such normal geodesics are always local length minimizers. However, there can be curves that are local length minimizers, but not normal geodesics. Such curves then have to belong to a class called \emph{abnormal curves}; see \cite{Mon02} for details and for further background on sub-Riemannian geometry. 

Relative to the sub-Riemannian structure and the point $x \in\M$, let $\mathbf{Cut}_0(x)$ be the set of points not connected to $x$ by a unique, non-conjugate minimizing curve. We also define $\mathbf{Cut}_0(\M) = \{ (x,y) \, : \, y \in \mathbf{Cut}_0(x)\}$. We note that Lemma~\ref{cutlocus} also holds for the sub-Riemannian cut locus, and that it is again symmetric.

\subsection{Sub-Riemannian manifolds as limits of Riemannian manifolds}
Let $(\M,g)$ be a Riemannian manifold and let $\Ho$ be a subbundle of $T\M$ that is bracket generating. Let $\V = \Ho^\perp$ denote the orthogonal complement and decompose the metric $g$ into a direct sum $g = g_\Ho \oplus g_\V,$
where $g_\Ho$ and $g_\V$ denote the restrictions of $g$ to respectively $\Ho$ and $\V$.
We define the canonical variation $g_\ve$ of $g$ such that for every $\ve > 0$,
\begin{equation} \label{gve} g_\ve = g_\Ho \oplus \frac{1}{\ve} g_\V.\end{equation}
We can then see the sub-Riemannian manifold $(M, \Ho, g_\Ho)$ as the limit as $\ve \downarrow 0$, in the way described below in Lemma~\ref{lemma:Converge}. See \cite[Section~2, Appendix~A]{BGMR19} for details and proof.

Relative to $g_\ve$, let $d_\ve$ be its Riemannian distance and with exponential map $\exp_\ve$. We define this exponential map on the cotangent bundle to the manifold, such that if $\sharp_\ve : T^*\M \to T\M$ is the identification of the cotangent bundle with the tangent bundle using $g_\ve$, then $\gamma(t) = \exp_\ve(tp)$, $p \in T^*\M$, is the $g_\ve$-geodesic with initial vector $\sharp_\ve p$.
\begin{lemma} \label{lemma:Converge} Assume that $(\M, g)$ is a complete manifold. Fix a point $x \in \M$.
\begin{enumerate}[\rm (a)]
\item Let $d_\ve$ be the distance of the metric $g_\ve$. Then $d_\ve \to d_0$ uniformly on compact sets as~$\ve \downarrow 0$. 
\item For $\ve_1 \geq \ve >0$, let $\gamma_\ve: [0,1] \to \M$, $t \mapsto \exp(tp_\ve)$, $p_\ve \in T^*_x\M$ a family of minimizing $g_\ve$-geodesics contained in a compact set with $\lim_{\ve \downarrow 0} \gamma_\ve(1) = y \not\in \mathbf{Cut}_0(x)$. Let $\gamma_0(t) = \exp_0(tp_0)$ be the unique sub-Riemannian geodesic from $x$ to $y$. Then $\gamma_\ve \to \gamma_0$ uniformly and $p_\ve \to p_0$ as~$\ve \downarrow 0$.
\item If $y \not\in \mathbf{Cut}_0(x)$, then there exists a neighborhood $U \ni y$ and an $\ve_2 >0$, such that $U \cap \mathbf{Cut}_\ve(x) = \emptyset$ for $0 \leq \ve < \ve_2$ and the map
$$(\ve, z) \mapsto r_\ve(z) = d_\ve(x,z),$$
is smooth on $[0,\ve_2) \times U$.
\item Let $\nabla f = \sharp_1 df$ denote the gradient of the metric $g$ and let $\nabla_\Ho f$ and $\nabla_\V f$ denote its projection to respectively $\Ho$ and $\V$. Then $\nabla_{\Ho} r_\ve \to \nabla_{\Ho} r_0$ and $\nabla_{\V} r_\ve \to \nabla_{\V} r_0$ as $\ve \downarrow 0$ uniformly on compact sets in $\M \setminus \mathbf{Cut}_0(x)$. In particular, $\| \nabla_{\Ho} r_\ve \|^2 \to 1$ uniformly on compact sets in $\M \setminus \mathbf{Cut}_0(x)$.
\end{enumerate}
\end{lemma}

\subsection{Sasakian manifolds} \label{sec:Sasakian}
Let $\theta$ be a non-vanishing one-form on a connected manifold $\M$ with $\Ho =  \ker \theta$. We call $\theta$ a \emph{contact form} if $d\theta|_{\wedge^2 \Ho}$ is non-degenerate. It follows that $\M$ is odd dimensional, that $\Ho$ has even rank and that it is bracket-generating. The \emph{Reeb vector field} is the unique vector field $Z$ satisfying
$$\theta(Z) = 1, \qquad d\theta(Z, \, \cdot \,) = 0.$$
We define $\V$ as the subbundle spanned by $Z$. There then exists a unique Riemannian metric $g$ and vector bundle map $J: T\M \to T\M$ such that
$$g(Z,X) = \theta(X), \qquad \langle J X, Y \rangle_g = d\theta(X,Y), \qquad J^2 X = - X + \theta(X) Z,$$
for any $X \in T\M$. We emphasize that $\Ho$ and $\V$ are orthogonal under this metric.

On contact manifolds $(\M, \theta, g)$, there is a also preferred choice of connection that preserves the decomposition $T\M = \Ho \oplus \V$, called the Tanno's connection, which was introduced in~\cite{Tanno}. It is the unique connection with torsion $T$ that satisfies:
\begin{enumerate}[\rm (i)]
\item $\nabla\theta=0$;
\item $\nabla Z=0$;
\item $\nabla g=0$;
\item ${T}(X,Y)=d\theta(X,Y)Z$ for any $X,Y\in \Gamma(\mathcal{H})$;
\item ${T}(Z,JX)= - JT(Z, X)$ for any vector field $X\in \Gamma(\mathcal{H})$.
\end{enumerate}
The manifold $(\M, \theta, g)$ is called \emph{$K$-contact} if $T(Z,\, \cdot \,) = 0$. This is equivalent to assuming that the Reeb vector field $Z$ is also Killing. It is called \emph{Sasakian} if it in addition satisfies $\nabla T = 0$.

\begin{remark}
If $(\M, \theta, g)$ is $K$-contact, then the Tanno connection coincides with the Bott connection. That is, if $\pi_{\calH}: T\M \to \calH$ is the orthogonal projection, $\nabla^g$ is the Levi-Civita connection of $g$, $Z$ is the Reeb vector field and $X,Y \in \Gamma(\calH)$, we can then describe $\nabla$ by
$$\nabla_X Y = \pi_{\calH} \nabla_X^g Y, \qquad \nabla_Z X = [Z,X], \qquad \nabla Z=0.$$
See e.g. \cite{BGKT} for details. If $\M$ is a strongly pseudo convex CR manifold with pseudo-Hermitian form $\theta$, then the Tanno's connection is the Tanaka-Webster connection. CR manifolds of K-contact type are Sasakian manifolds (see \cite{Dragomir}).
\end{remark}

If $\nabla$ is our Tanno connection on a Sasakian manifold $(\M, \theta, g)$, then its torsion $T$ is given by
$$T(X,Y) = \langle JX, Y \rangle_g Z.$$
Let $g_\ve$ be the canonical variation of $g$ as defined in \eqref{gve}. The geodesics of $g_\ve$ are in general not $\nabla$-geodesics since its torsion is not skew-symmetric for any $g_\ve$. For this reason, we also consider the connection
$$\hat \nabla_X^\ve Y = \nabla_X Y + \frac{1}{\ve}\theta(X) J Y,$$
with adjoint
$$\nabla^\ve_X Y = \nabla_X Y - \frac{1}{\ve} \theta(Y) JX - T(X,Y).$$
These connections are both compatible with $g_\ve$, and hence have skew-symmetric torsion. However, $\hat \nabla^\ve$ also preserves the decomposition $T\M = \Ho \oplus \V$. It is hence also compatible with $g_{\ve_2}$ for any other $\ve_2 > 0$ as well. Furthermore, if $\dot \gamma$ is the tangent vector of a geodesic of $g_\ve$, not only is this parallel with respect to $\hat \nabla^\ve$, but the same holds for its projection to $\Ho$ and $\V$, denoted by $\dot \gamma_\Ho$ and $\dot \gamma_\V$, respectively.

If $\hat T^\ve$ and $\hat R^\ve$ denote the torsion and the curvature, respectively, of $\hat \nabla^\ve$, then
$$\hat T^\ve(X, Y) = \langle JX, Y \rangle Z + \frac{1}{\ve} (\theta(X) J Y - \theta(Y) J X),$$
while from \cite{BGKT}, we have
\begin{align*}
\hat{R}^\varepsilon (X,Y)W =R(X,Y)W +\frac{1}{\varepsilon} \langle J X, Y \rangle_g JW.
\end{align*}

\begin{remark}\label{Rmk:SasakianCut}
In a Sasakian space, for every non-vanishing horizontal vector field $X$,  $T\M$ is always generated by $[X,\Ho]$ and $\Ho$. Therefore the sub-Riemannian structure on a Sasakian foliation is fat.  All sub-Riemannian geodesics are thus normal. See~\textup{\cite{RS}} for a detailed discussion of such structures. Furthermore, from Corollary 6.1 in \textup{\cite{RT}}, for every $x_0 \in \M$, the distance function $x \to r_0 (x)$ is locally semi-concave in $\M\setminus \{ x_0 \}$, and hence twice differentiable almost everywhere. Also, from Corollary 32 in \textup{\cite{BR2}}, $x \neq x_0$ is in $\mathbf{Cut}_0 (x_0)$ if and only if $r_0$ fails to be semi-convex at $x$. Therefore, $\mathbf{Cut}_0 (x_0)$ has measure 0.
\end{remark}

\section{Parallel and mirror maps on Sasakian manifolds with comparison results} \label{sec:MainSection}
\subsection{Hessian comparison theorem for Sasakian manifold}
We will first state a Hessian comparison theorem on Sasakian manifolds found in \cite{BGKT,BGMR19}, which uses notation needed later in the paper. Let $R$ be the curvature tensor of the Tanno connection $\nabla$. If $\Sec$ denotes the sectional curvature relative to $\nabla$, we introduce a $2$-tensor $\mathbf{K}_{\Ho, J}$ such that for any $w \in \calH \setminus 0$,  
\[
\mathbf{K}_{\mathcal{H},J} (w,w) = \| w\|_g^2 \,  \Sec( \spn\{ w, Jw\}).
\]
The quantity $\mathbf{K}_{\mathcal{H},J}$ is sometimes called the pseudo-Hermitian sectional curvature of the Sasakian manifold, which can be seen as the CR-analog of the holomorphic sectional curvature of a K\"ahler manifold \cite{B}. By removing this sectional curvature from the sum in the Ricci curvature, we define
\[
\mathbf{Ric}_{\mathcal{H},J^\perp} (w,w) =\mathbf{Ric}_\mathcal{H} (w,w)-\mathbf{K}_{\mathcal{H},J} (w,w), \qquad \mathbf{Ric}_\Ho(w,w) = \mathrm{tr}_\Ho \langle R(w, \times) \times, w \rangle_g.
\]
For any $\ve \geq 0$, let $r_{\ve,x}(y) = d_\ve(x,y)$ and consider the subset
\begin{equation} \label{Sigmave} \Sigma_\ve = \M \times \M \setminus \mathbf{Cut}_\ve(\M),\end{equation}
We define functions $h_\ve, v_\ve: \Sigma_\ve \to \mathbb{R}$, by
\begin{equation} \label{hvve} h_\ve(x,y) = \| \nabla_\Ho r_{\ve,x}(y) \|_{\Ho}, \qquad v_\ve(x,y) =  \| \nabla_\V r_{\ve,x}(y) \|_{\V}.\end{equation}
Note that $h_\ve^2 + \ve v_\ve^2 = 1$. Next, for any $k \in \mathbb{R}$, let $\mathfrak{s}_k(t)$ denote the solution of the equation $\ddot y + k y = 0$ with initial condition $y(0) = 0$ and $\dot y(0) =1$. Write $\mathfrak{c}_k(t)$ for its derivative, which satisfies the same ODE with initial conditions $y(0) =1$ and $\dot y(0) =0$. In other words,
$$
\mathfrak{s}_k(t) = \left\{
{\footnotesize \begin{array}{ll}
 \frac{\sin \sqrt{k} t}{\sqrt{k}} & \text{if } k >0, \\ [4pt]
t & \text{if } k =0, \\ \frac{\sinh \sqrt{-k} t}{\sqrt{-k}} & \text{if } k < 0, \end{array} } \right. 
\qquad
\mathfrak{c}_k(t) = \left\{
{\footnotesize \begin{array}{ll}
 \cos \sqrt{k} t & \text{if } k >0, \\ [4pt]
1 & \text{if } k =0, \\ \cosh \sqrt{-k} t & \text{if } k < 0. \end{array}}
\right.$$
We can use these to introduce comparison functions
$$F_{\mathrm{Rie}}(r,k) =  \frac{d}{dr} \log |\mathfrak{s}_k(r)|, \qquad F_{\mathrm{Sas}}(r,k) = \frac{d}{dr} \log \frac{1}{k^2}|2- 2\mathfrak{c}_k(r) - kr\mathfrak{s}_k(r)|,$$
or, in more detail,
\begin{align*}
F_{\mathrm{Rie}}(r,k)  & = { \footnotesize \left\{ \begin{array}{ll} \frac{\sqrt{k} \cos \sqrt{k} r}{\sin \sqrt{k} r} & k> 0, \\ [4pt] \frac{1}{r} & k =0, \\ [4pt] \frac{\sqrt{|k|} \cosh \sqrt{|k|}r}{\sinh \sqrt{|k|} r}  & k < 0, \end{array} \right.} & \qquad
F_{\mathrm{Sas}}(r,k)  &= {\footnotesize \left\{ \begin{array}{ll} \frac{\sqrt{k}(\sin \sqrt{k} r - \sqrt{k} r \cos \sqrt{k} r) }{2 - 2 \cos \sqrt{k} r - \sqrt{k} r \sin \sqrt{k} r} & k > 0, \\ [4pt] \frac{4}{r} & k =0, \\ [4pt] \frac{\sqrt{|k|} ( \sqrt{|k|} r \cosh \sqrt{|k|} r - \sinh \sqrt{|k|} r) }{2 - 2 \cosh \sqrt{|k|} r + \sqrt{|k|} r \sinh \sqrt{|k|} r} & k <0. \end{array}\right.}
\end{align*}
The following result is found in \cite{BGKT,BGMR19} for the Hessian with respect to the Tanno connection~$\nabla$.
\begin{theorem}[Hessian comparison theorem] \label{th:Hessian} Let $(x,y) \in \Sigma_\ve$ be given and write $h = h_\ve(x,y)$, $v= v_\ve(x,y)$ and $r= r_{\ve,x}(y)$. Write $\gamma = \gamma_{\ve,x,y}:[0,r] \to \M$ for the length minimizing geodesic from $x$ to $y$. For constants $k_1$ and $k_2$ to be defined, write
\begin{align} 
\label{Kconstants} & K_1 = K_{1,\ve}(x,y) = k_1 h_\ve(x,y)^2 + v_\ve(x,y)^2, \qquad K_2 = K_{2,\ve}(x,y) = k_2 h_\ve(x,y)^2 + \frac{1}{4} v_\ve(x,y)^2.
\end{align}
We then have the following bounds.
\begin{enumerate}[\rm (a)]
\item If $u = \dot \gamma_{\calH}(r)$, then $\frac{1}{h^2} \nabla^2_{ u, u} r_{\ve,x}(y) \leq \frac{1-h}{r}$. Furthermore, if $\mathbf{K}_{\calH,J}(\dot \gamma_{\calH}(t), \gamma_{\calH}(t)) \geq k_1$, then
$$\frac{1}{h^2} \nabla^2_{J u, J u} r_{\ve,x}(y) \leq F_{\mathrm{Sas}}(r, K_1).$$
\item If $\Sec(\spn\{\dot \gamma_{\calH}(t), v\})  \geq k_2$ for any $v \in \Ho_{\gamma(t)} \setminus 0$, and $w \in \Ho_y$ is a unit vector orthogonal to~$\dot \gamma_{\calH}(t)$ and~$J\dot \gamma_{\calH}(t)$, then
$$\nabla^2_{w,w} r_{\ve,x}(y) \leq F_{\mathrm{Rie}}(r, K_{2}).$$
\end{enumerate}
In particular, if we have global bounds
\begin{equation}
\label{CurvBoundSasakian}
 \mathbf{K}_{\mathcal{H},J}(w,w) \ge  k_1, \qquad \mathbf{Ric}_{\mathcal{H},J^\perp}(w,w) \ge (n-2)k_2, \qquad w \in \ch, \| w\|_g = 1,
\end{equation}
and if $\Delta_{\calH} = \tr_{\calH} \nabla^2_{\times,\times}$ is the sub-Laplacian of $\nabla$, then
$$\Delta_\Ho r_{\ve,x}(y) \leq  F_{\mathrm{Sas}}(r, K_1) + (n-2)  F_{\mathrm{Rie}}(r, K_2).$$
\end{theorem}

\subsection{Parallel map, mirror map, and their limit on Sasakian manifolds}\label{Sect:ParallelAndMirror}
Let $\Sigma_\ve$ be defined as the complement of $\mathbf{Cut}_\ve(\M)$ in $\M \times \M$ for any $\ve \geq 0$ as in \eqref{Sigmave}. We write $\pi_{1}, \pi_{2}: \Sigma_\ve \to \M$ for the respective projections $\pi_1(x,y) = x$ and $\pi_2(x,y) = y$.
For $\ve >0$, we can define a section $P_\ve \in \Gamma(\pi^*_1 T^* \M \otimes \pi_2^* T\M)$ such that for any $(x,y) \in \Sigma_\ve$, the map $P_\ve (x,y): T_x \M \to T_y \M$ denotes parallel transport along $\gamma_{\ve,x,y}$ with respect to~$\hat \nabla^\ve$. This map gives us a parallel transport that preserve the metric, as well as~$\Ho$ and~$\V$. Observe that for $\gamma = \gamma_{\ve, x,y}$ and $r= r_{\ve,x}(y)$
$$P_\ve(x,y) \dot \gamma_\Ho(0) = \dot \gamma_\Ho(r), \qquad P_\ve(x,y)\dot \gamma_\V(0) = \dot \gamma_\V(r).$$
For $\ve = 0$, we define the linear map $P_0(x,y) : T_x \M \to T_y \M$, $(x,y) \in \Sigma_0$ such that $P_{0}(x,y) w\mapsto X^w(r_{0,x}(y))$ where $X^w$ is the vector field along $\gamma_{0,x,y}$ solving
$$\nabla_{\dot \gamma_{0,x,y}} X^w(t) + v_0(x,y) JX^w(t) = 0, \qquad X^w(0) = w.$$
In particular, for $\gamma = \gamma_{0,x,y}$ and $r = r_{0,x}(y)$, we have
$$P_0(x,y) \dot \gamma(0) = \dot \gamma(r), \qquad P_0(x,y) Z_x = Z_y, \qquad P_0(x,y) J = JP_0(x,y),$$
and $P_0(x,y)$ maps $\ch_x$ onto $\ch_y$ isometrically.

For the construction of the mirror map, define $\Sigma_\ve' = \{ (x,y) \in \Sigma_\ve \, : h_\ve(x,y) >0\}$ for $\ve \geq 0$. Observe that $\Sigma_0 = \Sigma_0'$ since $h_0$ is identically $1$. For $(x,y) \in \Sigma_\ve'$ define $M_\ve(x,y): T_x\M \to T_y \M$ such that $M_\ve(x,y) w = P_\ve(x,y) w$ (resp. $-P_\ve(x,y)w$) if $w$ is orthogonal (resp. parallel) to $(\dot\gamma_{\ve,x,y}(0))_\Ho$.

\begin{lemma} \label{lemma:SigmaPrime}
For any $(x,y) \in \Sigma_0$, there is a neighborhood $U\ni (x,y)$ and some $\ve_2 >0$, such that $U\subseteq \Sigma_\ve'$ for $0 \leq \ve < \ve_2$ and we have $P_\ve \to P_0$ and $M_\ve \to M_0$ uniformly as $\ve \downarrow 0$.
\end{lemma}
 
\begin{proof}
It is sufficient to complete the proof for $M_\ve$. Let $(x,y) \in \Sigma_0$ be an arbitrary pair of points. From Lemma~\ref{lemma:Converge}, we know that for any $\tilde x \in \M$, there is a relatively compact neighborhood $
\tilde U_{\tilde x}$ of $y$ and a constant $\ve_{\tilde x} > 0$ such that $\tilde U_{\tilde x} \subseteq \M \setminus \mathbf{Cut}_\ve(\tilde x)$ for any $0 \leq \ve_{\tilde x}$ and such that $d_\ve(\tilde x, \, \cdot \,)$ converges uniformly. Let $W$ be a relatively compact neighborhood of $x$. If we define
$$\ve_2 = \min_{\tilde x \in W} \ve_{\tilde x}.$$
then $U = \{ (\tilde x,\tilde y) \, : \,  \tilde x \in W, \tilde y \in U_{\tilde x}\} \subseteq \Sigma_\ve$ for any $0 \leq \ve \leq \ve_2$. Since $h_0(x,y) = 1$, and by possibly shrinking $\ve_2$ and $U_{\tilde x}$, we can assume that $U \subseteq \Sigma'_{\ve}$ for any $0 \leq \ve \leq \ve_2$.

For $0 \leq \ve \leq \ve_2$ and $(\tilde x, \tilde y) \in U$, we define $p_\ve \in T^*_{\tilde x} \M$ such that $\exp_\ve (tp_\ve) = \gamma_{\ve, \tilde x, \tilde y}(t) = \gamma_{\ve,\Ho}(t)$ is a $g_\ve$-minimizing geodesic. We then note that
$$M_\ve: \dot \gamma_{\ve, \Ho}(0) \mapsto -\dot \gamma_{\ve, \Ho}(r_{\ve, x}(y)), \quad  \dot \gamma_{\ve, \Ho}(0) \mapsto J\dot \gamma_{\ve, \Ho}(r_{\ve, x}(y)), \quad Z_{\tilde x} \mapsto Z_{\tilde y}.$$
If $w \in \Ho_{\tilde x}$ is orthogonal to $\spn\{ \dot \gamma_{\ve,\calH}(0), J\dot \gamma_{\ve,\calH}(0) , Z_{\tilde x}\}$, we have that $M_\ve(x,y): w \mapsto X^w_\ve(r_{\ve,x}(y))$, where $X^w_\ve(t)$ is the result of $\hat \nabla^\ve$-parallel transport along $\gamma_{\ve}$. Since $\dot \gamma_{\ve, \V}(t) = \ve v Z(t)$, $v = v_\ve(\tilde x, \tilde y)$, we have that
$$0 = \hat \nabla_{\dot \gamma}^\ve X^w = \nabla_{\dot \gamma} X^w+ \frac{1}{\ve} \theta(\dot \gamma_\ve) J X^w = \nabla_{\dot \gamma} X^w+ v J X^w.$$
The result now follows.
 \end{proof}

\subsection{Jacobi fields and Sasakian models}
In what follows, we need to consider Jacobi fields and approximate solutions to the Jacobi equations. For the rest of the paper, if $Y$ is a vector field along $\gamma$, we will simply write $Y'$ for the covariant derivative $\hat \nabla^\ve_{\dot \gamma} Y(t)$. We will also identify vectors with their corresponding $\hat \nabla^\ve$-parallel vector fields in the notation. If $Y$ is a Jacobi field along a unit speed geodesic $\gamma$, then its defining equation is given by
\begin{align} \nonumber
& 0  = \hat{\nabla}^\ve_{\dot \gamma} \nabla_{\dot \gamma}^\ve Y - \hat{R}^\ve(\dot \gamma, Y) \dot \gamma \\ \label{RealJacobi}
& =  Y'' - \langle J \dot \gamma_\Ho, Y' \rangle_g Z - \frac{1}{\ve} \theta(\dot \gamma) J Y'+ \frac{1}{\ve} \theta(Y') J  \dot \gamma_\Ho  
 - R(\dot \gamma_\Ho, Y) \dot \gamma_\Ho - \frac{1}{\ve} \langle J \dot \gamma_\Ho,Y \rangle_g J  \dot \gamma_\Ho .
\end{align}
With $X$ horizontal and orthogonal to $\dot \gamma_\Ho$ and $J\dot \gamma_\Ho$, write
$$Y = c \dot \gamma   + \frac{a}{h} J \dot \gamma_\Ho -b \left( Z - \frac{ v}{h^2} \dot \gamma_\Ho \right) + X.$$
 Inserting this form into equation \eqref{RealJacobi}, we obtain
\begin{align*}
0 & = \ddot c \dot \gamma   + \frac{\ddot a}{h} J \dot \gamma_\Ho + \ddot b \left( Z - \frac{ v}{h^2} \dot \gamma_\Ho \right) + X'' - h\dot a Z - \frac{v}{h^2} (h^2 \dot c - v \dot b)  J \dot \gamma_\Ho  \\ 
& \qquad  + \frac{v}{h} \dot a  \dot \gamma_\Ho - v J X'+ \frac{1}{\ve} (\dot b + \ve v \dot c) J  \dot \gamma_\Ho    - \frac{1}{\ve} ha J  \dot \gamma_\Ho- R\left(\dot \gamma_{\calH}, \frac{a}{h} J \dot \gamma_{\calH}+ X \right)\dot \gamma_\calH. \end{align*}
If we consider these equations in the constant curvature case where
\begin{equation*} \label{ConstantCurvature} R(\dot \gamma_\Ho, J\dot \gamma_\Ho) \dot \gamma_\Ho = - h^2 k_1 J\dot \gamma_\Ho, \qquad R(\dot \gamma_\Ho, X) \dot \gamma_\Ho = -h^2 k_2  X, \end{equation*}
and with $K_1 = h^2 k_1 + v^2$ and $K_2 = h^2 k_2 + \frac{1}{4} v^2$, our equations become
\begin{equation} \label{JacobiEqs}
\left\{ \begin{aligned}
0 & = \ddot c \\
0 & = \ddot a + \frac{1}{h\ve} (\dot b- h a) + K_1 a , \\
0 & = \ddot b - h \dot a, \\
0 & = X'' - vJ X' + \left(K_2 - v^2/4 \right) X.
\end{aligned} \right. \end{equation}
These are the model equations for Jacobi fields. We consider a vector field $\hat Y = \hat Y^{a_0,a_1,u_0,u_1,K_1,K_2}$, of the form
$$\hat Y = \frac{a}{h} J \dot \gamma_\Ho -b \left( Z - \frac{ v}{h^2} \dot \gamma_\Ho \right) + X,$$
solving the equation \eqref{JacobiEqs} with $a(0) = a_0$, $a(r) = a_1$, $X(0) = u_0$ and $X(r) = P_\ve u_1$.

To simplify notation, we introduce the following conventions. If $u_0 \in \calH_x$ is a vector, we will use the same symbol for the corresponding $\hat \nabla^\ve$-parallel vector field along $\gamma = \gamma_{\ve,x,y}$. If $z \in \mathbb{C}$ is a complex number, we use the convention $z \cdot u_0 = \mathrm{Re}(z) u_0 + \mathrm{Im}(z) Ju_0$.

\begin{lemma} \label{lemma:Jacobi}
The vector field $\hat Y = \hat Y^{a_0,a_1,u_1,u_2,K_1,K_2} = \frac{a}{h} J \dot \gamma_\Ho -b \left( Z - \frac{ v}{h^2} \dot \gamma_\Ho \right) + X$ is given by
\begin{align*}
b(t) &= h \int_0^t a(s) \, ds - \frac{t}{r} \int_0^r a(s) \, ds \\
a(t) & = \frac{1}{\varphi(r)} (a_1 \varphi(t) + a_0 \varphi(r-t)), \\
X(t) & = \frac{z_{-v}(r-t)}{z_{-v}(r)} u_0 + \frac{z_v(t)}{z_{v}(r)} u_1,
\end{align*}
where
\begin{align*}
\varphi(t) &=   -K_1(r-\mathfrak{s}_{K_1}(r)-\ve hK_1 r ) \mathfrak{s}_{K_1}(t) + (1- \mathfrak{c}_{K_1}(r))  (1 - \mathfrak{c}_{K_1}(t)), & \qquad z_v(t) & = e^{iv/2} \mathfrak{s}_{K_2}(t).
\end{align*}
\end{lemma}

\begin{proof}
Solving equations \eqref{JacobiEqs} with our given initial conditions, since $b(0) = b(r) = 0$, it follows that
$b = h \int_0^t a(s) \, ds - \frac{t}{r} \int_0^r a(s) \, ds$.
Hence, we are reduced to the equations
\begin{equation} \left\{ \label{JacReduced} \begin{aligned}
\frac{1}{rh\ve} \int_0^r a(s)ds  & = \ddot a +  K_1 a , \\
0 & = X'' - vJ X' + \left(K_2 - v^2/4 \right) X.
\end{aligned} \right. \end{equation}
We first consider the solution of these equations with $a_0 =0$ and $u_0 =0$. By direct computation, we see that $X = \frac{z_v(t)}{z_v(r)} u_1$. Furthermore, using that $\int_0^r \mathfrak{s}_k(s) \, ds = \frac{1- \mathfrak{c}_k(r)}{k}$ and that $\int_0^r \mathfrak{s}_k(s) ds = \mathfrak{c}_k(r)$, we obtain that $a(t) = a_1 \frac{\varphi(t)}{\varphi(r)}$.
Observing how the equation \eqref{JacReduced} behaves under time-reversal, the result follows.
\end{proof}
Recall the geometric identities from Section~\ref{sec:Sasakian}. We define new comparison functions
$$G_{\mathrm{Rie}}(r, k)  = 2 \frac{d}{dr} \log |\mathfrak{c}_k(r/2)|, \qquad G_{\mathrm{Sas}}(r, k) = 2 \frac{d}{dr} \log \frac{1}{|k|} |\mathfrak{s}_k(r/2) - r/2 \cdot\mathfrak{c}_k(r/2)|.$$
or in more detail,
\begin{align*} 
G_{\mathrm{Rie}}(r, k)  &= {\footnotesize \left\{ \begin{array}{ll}
- \sqrt{ k} \tan \frac{\sqrt{k}r}{2} & \text{if } k > 0, \\ [4pt]
0 &  \text{if } k=0, \\ [4pt]
\sqrt{|k|} \tanh \frac{\sqrt{|k|}r}{2} & \text{if } k < 0,
\end{array}\right.} &
G_{\mathrm{Sas}}(r, k)  &= {\footnotesize \left\{ \begin{array}{ll}
\frac{r  k \tan \frac{\sqrt{ k} r}{2}}{2 \tan \frac{\sqrt{ k} r}{2} - r \sqrt{ k} } & \text{if } k > 0, \\ [8pt]
\frac{6}{r}  &  \text{if } k=0,\\ [4pt]
\frac{r  |k| \tanh \frac{\sqrt{ |k|} r}{2}}{r \sqrt{ |k|} - 2 \tanh \frac{\sqrt{ |k|} r}{2}} & \text{if } k < 0. 
\end{array}\right.} \end{align*}

\begin{lemma} \label{lemma:JacIndex}
Inserting $\hat Y = \hat Y^{a_0,a_1,u_1,u_2,K_1,K_2}$ into the index form $I_\ve$, we obtain
\begin{align*}
I_\ve(\hat Y, \hat Y) 
& =  (a_1^2 + a_0^2)F_{\Sas}(r, K_1)   - 2a_0 a_1 \left(F_{\Sas} (r,K_1) -  G_{\Sas}(r,K_1) \right) \\
& \qquad + (  \|u_0\|^2 + \|u_1\|^2) G_{\Rie}(r,K_2) + \frac{1}{\mathfrak{s}_{K_2}(r)} \left( \|u_0\|^2 + \|u_1\|^2 - 2\langle u_1, e^{irv/2} u_0 \rangle \right) \\
& \qquad - \int_0^r \langle R(\dot \gamma_{\calH}, \hat Y) \hat Y, \dot \gamma_\calH \rangle \, dt + (K_1- v^2) \int_0^r a(t)^2 \, dt - \left( K_2- \frac{v^2}{4} \right) \int_0^r \| X \|^2 \, dt .
\end{align*}
In particular, if $a_0 = a_1$, $u_0 = u_1$ and we have curvature bounds
\begin{equation*} \label{Bound1} \begin{aligned}
\langle R(\dot \gamma_\Ho, J\dot \gamma_\Ho) J\dot \gamma_{\calH}, \dot \gamma_\Ho \rangle_g & \geq h^4 (k_1+k_3), \\
\langle R(\dot \gamma_\Ho, X)X, \dot \gamma_\Ho \rangle_g & \geq h^2 (k_2+k_3)  \|X\|^2,\\
|\langle R(\dot \gamma_\Ho, X)J\dot \gamma_{\calH}, \dot \gamma_\Ho \rangle_g| &\leq h^3 k_3  \|X\|, \end{aligned}
\end{equation*}
then
\begin{align*}
I_\ve(\hat Y, \hat Y) & \leq  2a_1^2 G_{\Sas}(r,K_1)  + 2 \|u_1\|^2 G_{\Rie}(r,K_2) .
\end{align*}

\end{lemma}

\begin{proof}
Inserting $\hat Y$ into the index form, we obtain
\begin{align*}
I_\ve(\hat Y, \hat Y) & = \langle \hat Y(r), \hat Y'(r) \rangle_\ve - \langle \hat Y(0), \hat Y'(0) \rangle_{\ve}
- \int_0^r \langle \hat Y, \hat \nabla_{\dot \gamma}^\ve \nabla_{\dot \gamma}^\ve \hat Y - \hat R^\ve(\dot \gamma, \hat Y) \dot \gamma \rangle_\ve \, dt \\
& = a_1 \cdot \dot a(r) -a_0 \cdot \dot  a(0) + \langle u_0, X'(r) \rangle - \langle u_1, X'(0) \rangle \\
& \qquad - \int_0^r \langle R(\dot \gamma_{\calH}, \hat Y) \hat Y, \dot \gamma_\calH \rangle \, dt + (K_1- v^2) \int_0^r a(t)^2 \, dt - K_2 \int_0^r \| X \|^2 \, dt .
\end{align*}
We now compute that
$$\varphi(r)  = 2 - 2\mathfrak{c}_{K_1}(r) -K_1r \mathfrak{s}_{K_1}(r)  + \ve hK_1^2 r\mathfrak{s}_{K_1}(r),$$
\begin{align*}
\dot \varphi(0) & = - K_1(r-\mathfrak{s}_{K_1}(r)-\ve hK_1 r )  & \dot a(0) & =  \frac{1}{\varphi(r)} (a_1 \dot \varphi(0) - a_0 \dot \varphi(r) ), \\
\dot \varphi(r) & =K_1( \mathfrak{s}_{K_1}(r)- r\mathfrak{c}_{K_1}(r)  + \ve hK_1 r \mathfrak{c}_{K_1}(r)  )  & \dot a(r) & =  \frac{1}{\varphi(r)}( a_1 \dot \varphi(r) - a_0 \dot \varphi(0) )  , \\
X'(0) & = \frac{iv}{2} u_0 - \frac{\mathfrak{c}_{K_2}(r)}{\mathfrak{s}_{K_2}(r)} u_0 + \frac{1}{z_{v}(r)} u_1, & 
X'(r) & =   - \frac{1}{z_{-v}(r)} u_0 + \frac{iv}{2} u_1 + \frac{\mathfrak{c}_{K_2}(r)}{\mathfrak{s}_{K_2}(r)} u_1.
\end{align*}
By taking derivatives in $\ve$, we see that
\begin{align*}
\frac{\dot \varphi(r)}{\varphi(r)} & = 
\frac{K_1( \mathfrak{s}_{K_1}(r)- r\mathfrak{c}_{K_1}(r)  + \ve hK_1 r \mathfrak{c}_{K_1}(r)  )}
{2 - 2\mathfrak{c}_{K_1}(r) -K_1r \mathfrak{s}_{K_1}(r)  + \ve hK_1^2 r\mathfrak{s}_{K_1}(r) } \\
& \leq 
\frac{K_1( \mathfrak{s}_{K_1}(r)- r\mathfrak{c}_{K_1}(r)  )}
{2 - 2\mathfrak{c}_{K_1}(r) -K_1r \mathfrak{s}_{K_1}(r)   } = F_{\Sas}(r,K_1) \\
\frac{\dot \varphi(r) - \dot \varphi(0)}{\varphi(r)} &
= \frac{K_1r (1- \mathfrak{c}_{K_1}(r))( 1 - \ve h K_1 )}
{2 - 2\mathfrak{c}_{K_1}(r) -K_1r \mathfrak{s}_{K_1}(r)  + \ve hK_1^2 r\mathfrak{s}_{K_1}(r) } \\
& \leq \frac{K_1r (1- \mathfrak{c}_{K_1}(r))}
{2 - 2\mathfrak{c}_{K_1}(r) -K_1r \mathfrak{s}_{K_1}(r) } = G_{\Sas}(r,K_1).
\end{align*}
We furthermore see that
\begin{align*}
& \langle u_1, X'(r) \rangle - \langle u_0, X'(0) \rangle \\
& =(  \|u_0\|^2 + \|u_1\|^2) \frac{\mathfrak{c}_{K_2}(r)}{\mathfrak{s}_{K_2}(r)} + \frac{1}{\mathfrak{s}_{K_2}(r)}\left( \langle u_1, - e^{irv/2} u_0 \rangle_g - \langle u_0, e^{-i rv/2} u_1 \rangle_g \right)  \\
& =(  \|u_0\|_g^2 + \|u_1\|_g^2) F_{\Rie}(r,K_2) - \frac{2}{\mathfrak{s}_{K_2}(r)}\langle u_1, e^{irv/2} u_0 \rangle_g \\
& =(  \|u_0\|^2 + \|u_1\|^2) G_{\Rie}(r,K_2) + \frac{1}{\mathfrak{s}_{K_2}(r)} \left( \|u_0\|^2 + \|u_1\|^2 - 2\langle u_1, e^{irv/2} u_0 \rangle \right) .
\end{align*}
The result follows. \end{proof}

\subsection{Index form relative to the parallel transport map}
Let $\Sigma_\ve$ be defined as in Section~\ref{Sect:ParallelAndMirror}. For any $(x,y) \in \Sigma_\ve$, we define
\begin{equation} \label{IDef} I_{\ve}( x, y) = \sum_{i=1}^{n} I_\ve(Y_i, Y_i) \end{equation}
where $Y_i$ are Jacobi fields along $\gamma_{\ve,x,y}$ such that $Y_1(0), \dots, Y_{n}(0)$ is an orthonormal basis of $\Ho_x$ and $Y_i(d_\ve(x,y)) = P_{\ve}(x,y) Y_i(0)$. We want a way to bound this function using the curvature of the Tanno connection $\nabla$.

\begin{lemma} \label{lemma:BoundSas}
Assume that for some $k_1$ and $k_2$, the bounds in \eqref{CurvBoundSasakian} hold. Let $(x,y) \in \Sigma_\ve$ be fixed. Then for $K_1$ and $K_2$ as in \eqref{Kconstants},
\begin{align*}
I_\ve(x,y) & \leq 2 G_{\mathrm{\Sas}}(r, K_1) + 2 (n-2) G_{\mathrm{Rie}}(r, K_2).
\end{align*}
\end{lemma}

\begin{proof}
Let $\gamma$ be the unique unit speed geodesic from $x$ to $y$, and use $r = d_\ve(x,y)$, $h = h_\ve(x,y) = \| \dot \gamma_\calH\|_g$ and $v = v_\ve(x,y)= \ve^{-1} \| \dot \gamma_\V \|_g$. If $h = 0$, we can choose $Y_1, \dots, Y_n$ as an orthonormal basis of $\hat \nabla^\ve$-parallel vector fields along $\gamma$. These will not be Jacobi fields, but we will still have
$$\textstyle I_\ve(x,y) \leq \sum_{i=1}^n I_\ve(Y_i, Y_i) = 0,$$
and the theorem holds true in this case. For the remainder of the proof we will assume that $h >0$.

First, we define $Y_1(t) = \frac{1}{h} \dot \gamma_\Ho = \frac{1}{h (1+ \ve v^2)} \dot \gamma - \frac{\ve v}{h (1+ \ve v^2)} \left(Z - \frac{v}{h^2} \dot \gamma_\Ho \right)$. Then $I_\ve(Y_1, Y_1) = 0$.
The remaining elements of the basis will be on the form $\hat Y = \hat Y^{a_0,a_1,u_0,u_1, K_1,K_2}$ as in Lemma~\ref{lemma:Jacobi} with $K_1 = h^2 k_1+ v^2$ and $K_2 = h^2 k_2 + \frac{1}{4} v^2$. For $Y_2$, we choose $a_0 = a_1 =1$ and $u_0 = u_1=0$, so that
$$I_\ve(Y_2,Y_2) \leq G_{\Sas}(r,K_1).$$
For the other terms, let $w_3, \dots, w_n$ be a choice of orthonormal basis of the complement to $\dot \gamma$ and~$J\dot \gamma$. We then consider $Y_j$ on the previously mentioned from with $a_0 = a_1 = 0$ and $u_0 = u_1 = w_j$ and obtain
$$
\textstyle \sum_{j=3}^{n}I_\ve(Y_j, Y_j) \leq 2 G_{\Rie}(r,K_2) .
$$
Both of these inequalities follow from bounds~\eqref{CurvBoundSasakian} which combined with Lemma~\ref{lemma:JacIndex} completes the proof.
\end{proof}

\begin{remark}
We note the analogy of the Riemannian and Sasakian comparison functions in the Hessian bounds and the bounds for $I_\ve$, in that we have relations
$$\frac{d}{dr} \mathfrak{s}_k(r) = \mathfrak{c}_k(r),  \qquad \frac{d}{dr} \frac{1}{k^2}(2- 2\mathfrak{c}_k(r) - k r\mathfrak{s}_k) = \frac{1}{k} (\mathfrak{s}_k(r)- r \mathfrak{c}_k(r)).$$
Observe also that, in contrast to $G_{\mathrm{Rie}}(r,k)$, the function $G_{\mathrm{Sas}}(r,k)$ is always nonnegative for any $k$.
\end{remark}

\subsection{Expansion of distance along geodesic}
Consider $\M$ with the metric $g_\ve$ and let $\eta(t)$ be a $g_\ve$-geodesic. Observe that we have geodesic equation
$$0 = \hat \nabla_{\dot \eta}^{\ve} \dot \eta(t) = \nabla_{\dot \eta} \dot \eta(t) + \frac{1}{\ve} \theta(\dot \eta(t)) J\dot \eta(t) .$$
If $\dot \eta(0) \in \calH$, then it follows that $\eta(t)$ is the solution of $\nabla_{\dot \eta} \dot \eta =0$ and in particular will remain tangent to $\calH$ for all time. Furthermore, as this equation is independent of $\ve$, $\eta(t)$ will be a $g_\ve$ for any $\ve$.
This includes the case $\ve =0$ as $\eta(t) = \exp_0(t \psi)$ with $\psi$ being the unique covector in $T_{\eta(0)}^*\M$ vanishing on $\calV$ and satisfying $\sharp_\ve \psi = \sharp_0 \psi = \dot \eta(0)$. We remark that a sub-Riemannian Brownian motion can be constructed as a random walk of such geodesics, see \cite{BosNeelRiz} or \cite{GordLaet} for details. We want to understand how the sub-Riemannian metric will change with respect steps along such geodesics.

Let $(x,y)$ be any pair of points in $\M \times \M$ and let $\gamma(t) = \exp(tp)$ be a length-minimizing geodesic from $x$ and $y$,
parametrized by arc length. Write $p_Z = p(Z|_x)$. Similar to earlier, we define $P_0(\gamma): T_x\M \to T_y \M$ by parallel transport corresponding to the operator
$$X \mapsto \nabla_{\dot \gamma} X + p_Z J X,$$
but with this definition independent of whether or not $(x,y)$ is in $\mathbf{Cut}(\M)$.
We use the same symbol for the map $P_0(\gamma): T^*_x\M \to T^*_y\M$, $\alpha \mapsto P(\gamma)^* \alpha$.
Define $M_0(\gamma)$ analogously.
\begin{theorem} \label{theorem:Expansion}
Let $\psi_0, \psi_1 \in T^*_x \M$ be any pair of covectors satisfying $\langle \psi_0 , Z_x \rangle_g = \langle  \psi_1 , Z_x \rangle_g = 0$. For $j=0,1$, write
\begin{equation}
\label{DecompPsi} \sharp_0 \psi_j = c_j \dot \gamma(0) + a_j J \dot \gamma(0) + u_j,\end{equation}
with $u_j$ being orthogonal to $\dot \gamma(0)$ and $J\dot \gamma(0)$. Finally, let $R$ be the curvature of the Tanno connection $\nabla$ and define
\begin{align*}
k_1 +k_3 & = \min_{0\leq t \leq r} \Sec( \spn\{ \dot \gamma(t) , J\dot \gamma(t) \}), \\
k_2& = \min_{0\leq t \leq r} \min_{\begin{subarray}{c} u \in \ch_{\gamma(t)}, \| u\|_g =1 \\  \langle u, \dot \gamma(t) \rangle_g = \langle u , J \dot \gamma(t) \rangle_g =0 \end{subarray} } \Sec( \spn\{ \dot \gamma(t) , u \}), \\
k_3& = \max_{0\leq t \leq r} \max_{\begin{subarray}{c} u \in \ch_{\gamma(t)}, \| u\|_g =1 \\  \langle u, \dot \gamma(t) \rangle_g = \langle u , J \dot \gamma(t) \rangle_g =0 \end{subarray} } |\langle R(\dot \gamma(t), u) J \dot \gamma(t), \dot \gamma(t) \rangle_g|.
\end{align*}
Define $K  = k_1  + p_Z^2$ and $K_2 = k_2+ p_Z^2/4$.
We then have bounds
\begin{align*}
& d_0(\exp(t\psi_0) , \exp(t P_{0}(\gamma) \psi_1))  - r - t  (c_1-c_0)\\
& \leq \frac{t^2}{2} (a_1^2 + a_0^2)F_{\Sas}(r, K_1)   - t^2 a_0 a_1 \left(F_{\Sas} (r,K_1) -  G_{\Sas}(r,K_1) \right) \\
& \qquad + \frac{t^2}{2} (  \|u_0\|^2 + \|u_1\|^2) G_{\Rie}(r,K_2) + \frac{t^2}{2\mathfrak{s}_{K_2}(r)} \left( \|u_0\|^2 + \|u_1\|^2 - 2\langle u_1, e^{irv/2} u_0 \rangle \right)+ O(t^3).
\end{align*}
\end{theorem}

\begin{proof} 
Without loss of generality, we may assume that $(x,y) \in \Sigma_0$. If not, we may partition $\gamma$ into pieces without conjugate points and finish by using the triangle inequality.
This can always be done to avoid the cut locus when there are no abnormal geodesics, see \cite{ABR19} for details.

For sufficiently small $\ve_2 >0$, $(x,y)$ is not in the cut locus of $g_\ve$ for $\ve \leq \ve_2$ by Lemma~\ref{lemma:Converge}. Let~$\gamma_\ve$ be the unique $g_\ve$-geodesic from $x$ to $y$ with $P_\ve = P_\ve(x,y)$ denoting $\hat \nabla^\ve$-parallel transport along this curve. As before, write $h_\ve = \|\dot \gamma_{\ve,\ch}(0)\|_g$ and $v_\ve = \ve^{-1} \|\gamma_{\ve,\V}(0)\|_g$. For sufficiently small values of $\ve$, $h_\ve$ is positive by Lemma~\ref{lemma:SigmaPrime}. We introduce the covectors $\psi_{0,\ve}, \psi_{1,\ve} \in T^* \M$ such that,
$$\sharp_\ve \psi_j =  c_{j} \dot \gamma_{\ve}(0) + a_{j} J \dot \gamma_\ve(0) + u_{j},$$
By uniform convergence of geodesics, $\exp_\ve(t\psi_{j,\ve})$ converge to $\exp_0(t\psi_j)$.

We consider $f_\ve(t) = d_{\ve}(\exp(t \psi_{0,\ve}), \exp(t P_\ve \psi_{1,\ve}))$ with $r_\ve = f_\ve(0) = d_\ve(x,y)$. We know that $f_\ve(t)$ converges to $f_0(t)$ uniformly for any sufficiently small $t$.  We also know from Lemma~\ref{lemma:DistanceVariation} that
$$\frac{d}{dt} f_\ve(0) = \langle P_\ve \sharp_\ve \psi_{1,\ve},  \dot \gamma_\ve (r_\ve) \rangle_g -  \langle \sharp_{\ve} \psi_{1,\ve},  \dot \gamma_\ve (0) \rangle_g = \langle \psi_{1,\ve} - \psi_{0,\ve},  \dot \gamma_\ve (0) \rangle = c_{1} - c_0,$$
and furthermore
$$\frac{d^2}{dt^2} f_\ve(0) = I_\ve(\hat Y, \hat Y),$$
where $\hat Y = Y^\perp$ is the $g_\ve$-Jacobi field with initial conditions $Y(0) = a_{0} J \dot \gamma_\ve(0) + u_{0}$ and final condition $Y(r_\ve) = P_\ve (a_{1} J \dot \gamma_\ve(0) + u_{1})$. The result now follows from the Index Lemma and  Lemma~\ref{lemma:JacIndex}.
\end{proof}


\section{Application: Couplings on Sasakian manifolds} \label{sec:Coupling}
We now want to consider a coupling of diffusions with generators $\frac{1}{2} \Delta_\Ho$. Let $B_t = B_t(x)$ be the Brownian motion of the inner product space $\Ho_x$, $x \in \M$, defined on a probability space $(\Omega, \mathscr{F}_{\cdot}, \mathbb{P})$. We define a coupling of $Z_t^\ve(x,y) = (X_t(x), Y_t^\ve(y))$ by $X_0(x) = x$, $Y_0^\ve(y) = y$,
$$dX_t(x) = \hptr_t^{\ve} \circ dB_t, \qquad dY_t^\ve = P_{\ve} (X_t(x),Y_t^\ve(y))\hptr_t^{\ve} \circ dB_t,$$
where $\hptr_t^{\ve}: T_x M \to T_{X_t(x)}M$ denotes $\hat{\nabla}^{\ve}$-parallel transport along $X_t(x)$. This coupling process is defined up to a time
$$\tau_\ve = \inf \{ t >0 \, : \, (X_t, Y_t^\ve) \in \mathbf{Cut}_\ve(\M) \},$$
when the process hits the cut-locus. We then note the following consequence of our previous comparison result.
\begin{proposition} \label{prop:SasDistBound}
Define $\varrho_{t}^\ve =\varrho^\ve_t(x,y) = d_\ve(Z_t^\ve(x,y)) =  d_\ve(X_t(x),Y_t^\ve(y))$ and $Z_t^\ve = Z_t^\ve(x,y)$. Assume that the curvature bounds \eqref{CurvBoundSasakian}. Then for $t < \tau_\ve$, we have inequality
$$d\varrho_{t}^\ve \leq   G_{\mathrm{Sas}}(\varrho_\ve, K_1(Z_t^\ve)) dt +(n-2) G_{\mathrm{Rie}}(\varrho_\ve, K_2(Z_t^\ve)) dt,$$
\end{proposition}

\begin{proof} Recall that the rank of $\calH$ is $n$. Let $\pi: \mathrm{O}(\Ho) \to \M$ denote the orthonormal frame bundle bundle, i.e. the $\mathrm{O}(n)$-principal bundle where the fiber at $x \in \M$ consists of all linear isometries $\varphi: \mathbb{R}^n \to \Ho_x$. Let $\mathcal{E} = \mathcal{E}^\ve \subseteq T\mathrm{O}(\Ho)$ denote the subbundle defined by derivatives of all $\hat \nabla^\ve$-parallel frames along curves in $\M$. Then $\calE$ is a principal Ehresmann connection on $\pi$. For any $e \in \mathbb{R}^n$, let $H_{e}$ denote the vector field on $\mathrm{O}(\Ho)$ uniquely determined by the properties
$$H_e(\varphi) \in \calE_\varphi, \qquad \pi_* H_e(\varphi) = \varphi e.$$

We write the lift $\tilde Z_t^\ve = (\tilde X_t, \tilde Y^\ve_t)$ of $Z_t^\ve$ to $\mathrm{O}(\Ho) \times \mathrm{O}(\Ho)$ with respect to the two copies of $\mathcal{E}$. Write $\Pi: \mathrm{O}(\Ho) \times \mathrm{O}(\Ho) \to \M \times \M$ for the projection on the product. Define $H_{e,1}$ and $H_{e,2}$ analogues of $H_e$ but on respectively the first and the second component. If $e_1, \dots, e_n$ denotes the standard basis on $\mathbb{R}^n$, it follows that $\tilde Z_t^\ve$ is a solution of the SDE,
$$d\tilde Z_t^\ve = \sum_{i=1}^n  (H_{e_i,1} + H_{P_\ve e_i} ) \circ dB^i_t.$$
This gives us the following expression for It\^o differential of $\varrho_t^\ve$
$$d\varrho^\ve_{t} = \sum_{i=1}^n (H_{e_i,1} + H_{P_\ve e_i,2}) (d_\ve \circ \Pi)(\tilde X_t, \tilde Y_t^\ve) dB^i_t + \frac{1}{2} \sum_{i=1}^n (H_{e_i,1} + H_{P_\ve e_i,2})^2 (d_\ve \circ \Pi)(\tilde X_t, \tilde Y_t^\ve) dt$$

Using Lemma~\ref{lemma:DistanceVariation}, we can conclude that the martingale term vanishes, as for any $(x,y)$ the corresponding geodesic $\gamma= \gamma_{x,y}$ satisfies $P_{\ve,x,y}\dot \gamma_\ch(0) = \dot \gamma_\ch(r)$. These two observations also give us the conclusion that for $t < \tau_\ve$
$$d\varrho^\ve_{t}(x,y) = \frac{1}{2} I_{\ve} (X_t(x), Y_t(y)) dt.$$
The result now follows from Lemma~\ref{lemma:BoundSas}.
\end{proof}

\begin{corollary}
If the bounds \eqref{CurvBoundSasakian} holds with $k_1 = 0$ and $k_2\leq 0$, then for $t < \tau_\ve$,
$$\varrho_t^\ve(x,y) = d_\ve(X_t(x), Y_t^\ve(y)) \leq \sqrt{d_\ve(x,y)^2 e^{(n-2) |k_2| t} + \frac{12}{(n-2) |k_2|} \left(e^{(n-2) |k_2| t} - 1\right)},$$
where we interpret the expression above for $k_2 = 0$ as $\varrho_t^\ve \leq \sqrt{d_\ve(x,y)^2 + 12 t}$.
\end{corollary}

\begin{proof}
Write $h_{\ve,t} = h_\ve(X_t(x), Y_t^\ve(y))$. From Proposition~\ref{prop:SasDistBound} and using that $0 \leq h_{\ve,t} \leq 1$, we know that
\begin{align*}
d\varrho_t^\ve & \leq  \frac{6 }{ \varrho_t^\ve} dt + (n-2) \sqrt{h_{\ve,t}^2 |k_2| }\tanh \left( \frac{\sqrt{h_{\ve,t }^2 |k_2| } \varrho_t^\ve}{2} \right) dt \\
& \leq \frac{6}{ \varrho_t^\ve} dt + \frac{n-2}{2} |k_2| \varrho_t^\ve dt.
\end{align*}
Hence,
$$d(e^{-(n-2)|k_2|t}(\varrho_t^\ve)^2) \leq 12 e^{-(n-2)|k_2|t}  dt .$$
The result follows.
\end{proof}

\appendix


\end{document}